\documentclass[12pt]{article}

\usepackage{amsmath} 
\usepackage{amssymb}
\usepackage{amsthm}
\usepackage{graphicx}
\usepackage{amscd}
\usepackage{epic, eepic}
\usepackage{url}
\usepackage{xcolor}
\usepackage{todonotes}
\usepackage{enumerate}
\usepackage{enumitem}
\usepackage{comment}
\usepackage{hyperref}
\usepackage[utf8]{inputenc}
\usepackage{tabularray}
\newtheorem{theorem}{\bf Theorem}[section]
\newtheorem{lemma}[theorem]{\bf Lemma}
\newtheorem{cor}[theorem]{\bf Corollary}

\newtheorem{prop}[theorem]{\bf Proposition}

\theoremstyle{definition}

\newcommand{\Z}{\mathbb{Z}}

\newcommand{\dd}{\frac{p-1}{2}}
\newcommand{\ee}{\frac{p+1}{2}}

\newcommand{\F}{\mathbb{F}}

          \usepackage{array}

\usepackage{todonotes}

\title{Degree gap of polynomials of small range sum}

\author{ Ádám Markó \thanks{E\"otv\"os Lor\'and University, Institute of Mathematics, Budapest, Hungary
E-mail: {\tt marqadam@gmail.com}} \\ G\'abor Somlai \thanks{E\"otv\"os Lor\'and University, Institute of Mathematics, Budapest, Hungary
E-mail: {\tt gabor.somlai@ttk.elte.hu}\\Research supported by Bolyai scholarship, and OTKA 138596\\
Research was carried out as part of the thematic semester Fourier Analysis and Additive Problems at Erd\H{o}s Center}} %\todo{affiliáció/mail/támogató}

\date{}

\begin{document}
\maketitle
\begin{abstract}
Polynomials of degree $\frac{p-1}{2}$ of range sum $p$ were determined by the second authors relying on a joint work of the authors by Kiss and Nagy. We prove that for large enough primes there is no polynomial of degree $\frac{p+1}{2}$ of range sum $p$.
\end{abstract}

\section{Introduction}
We consider polynomials and polynomial functions over the field of degree $p$, where $p$ is a prime. It is well known that there is a natural bijection between polynomials of degree at most $p-1$ and functions from $\F_p$ to $\F_p$. On the other hand, $\F_p$ can be identified with the following subset $\{0,1, \ldots, p-1\}$ of the integers so these functions can also be considered  as functions from $\F_p$ to $\Z$.

Thus we obtain an  identification of the elements of $\F_p[x]$ of degree at most $p-1$ with functions from $\F_p$ to $\{0,1,\ldots, p-1\}$. This allows us to sum the elements of the range of a polynomial from  $\F_p[x]$, and obtain a nonnegative integer. 
It was proved in \cite{som} that if $f \in \F_p[x]$ is a polynomial of range sum $p$, then either $f=1$ or $\deg{f} \ge \frac{p-1}{2}$. The second author conjectured that up to affine transformations, the only polynomial with the preceding properties is $x^{\frac{p-1}{2}}+1$. The conjecture is clearly false, which is shown by the polynomial $\frac{p+1}{2}(x^{\frac{p-1}{2}}+1)$. It would have been more plausible to conjecture that there are no further polynomials since both of these two types appear as projection polynomial of the set defined by Lovász and Schrijver. This version of the conjecture was later confirmed.

The lower bound for the degree of polynomials whose range sum is equal to $p$, combined with a lemma from \cite{kissomlai} provide a new proof for Rédei's result on the number of directions determined by a set of $p$ points, which was an important result in finite geometry. The original proof of it is the first instance 
%one of the first appearance 
of polynomial methods used for discrete geometry problems.
It might be important to note that there were several other proof of Rédei's theorem, both of combinatorial origin \cite{dressklinmuzchuk} and one using Fourier techniques \cite{lev}. 

Lovász and Schrijver proved that there is a unique set of size $p$ in $\F_p^2$ determining exactly $\frac{p+3}{2}$ directions.
A further investigation of the connection between the range sum and the degree of a polynomial was proved in \cite{negyen}, where the authors verified that the only possible projection polynomials of degree exactly $\frac{p-1}{2}$ of a set of size $p$ all are derived from the Legendre symbol if $p$ is large enough. 
This result gave an alternative proof for the result of Lovász and Schrijver but only for the same set of primes, which are larger than $7.48*10^6$. This was further developed by the first author \cite{adam} to primes larger than 23\footnote{For small primes the same statement was verified both by the first author and Marcell Alexy. }.

Another important result in this field was proved by G\'acs \cite{gacscombinatorica}, who established another gap between $\frac{p+5}{2}$ and $\lfloor 2\frac{p-1}{2} \rfloor$ for the possible number of directions determined by a set of size $p$. We intend to find a similar gap for the degree of polynomials of range sum $p$, but it seems very difficult to achieve such a result so we prove the following. 

\begin{theorem}\label{thm1}
If $p$ is larger than 32, then there is no polynomial $f$ of degree $\frac{p+1}{2}$ such that $\sum_{x\in \F_p}f(x)=p$. 
\end{theorem}
In order to complement this theorem we construct polynomials for small primes $p$ of degree $\frac{p+1}{2}$ with range sum $p$ showing that polynomials of range sum $p$ of degree between $\frac{p-1}{2}$ and $\frac{2p}{3}$ probably exist. 

Further we construct polynomials of degree $2\frac{p-1}{3}$ of range sum $p$, which are not projection polynomials of the sets constructed by G\'acs.

\section{Notation}
Let $S$ be a subset of $\mathbb{F}_p^2$, where $p$ is a prime and $\mathbb{F}_p$ is the field of $p$ elements. We describe the set of \emph{directions} determined by $S$ in the following way. 
Let us consider the nonzero elements of $S-S$. For each nonzero vector in $\mathbb{F}_p^2$ we can assign an element of the projective line $PG(1,p)$ by considering two vectors equivalent if they are nonzero multiples of each other. 

We will treat the elements of $\mathbb{F}_p$ in two different ways. In some cases we identify them with the set $\{0,1, \ldots , p-1 \}$, which is a subset of the integers. We exploit this identification to talk about the range sum of a polynomial (function). Let $f$ be a polynomial in $\F_p[x]$. Every element $f(x)$ of the range can be considered as an element of $\{0,1, \ldots, p-1 \}$ so we may sum the range as integers. We will consider those polynomials where the sum of the range is equal to $p$ so we write $\sum_{x \in \F_p} f(x)=_{\Z} p$, indicating that the numbers we sum are elements of $\Z$. 

The Legendre symbol is denoted by $(\frac{a}{p})$ is equal to 1 if and only of $a$ is a quadratic residue $\pmod{p}$ and it is $-1$ if $a$ is a quadratic nonresidue, and $(\frac{0}{p})=0$.

We will rely on the results of \cite{negyen} so we first recall the essential lemmas that are needed to start the new investigation. 
\begin{lemma}
Let $f$ be a polynomial of degree $\frac{p+1}{2}$ of range sum $p$. Then $f$ has at least $\frac{p-3}{2}$ roots. 
\end{lemma}

The following equality is folclore and the proof is trivial. 
\begin{lemma}\label{lem:sumandleadingcoeff}
Let $\sum_{n=0}^{p-1}a_nx^n$ be a polynomial in $\mathbb{F}_p$. Then
    $$ \sum_{x \in \mathbb{F}_p} \sum_{n=0}^{p-1}a_nx^n \equiv -a_{p-1} \pmod{p}.
    $$
\end{lemma}
Another important information we can derive Proposition 3.2 in \cite{negyen} is that the number of different root of a polynomial of degree $\ee$ of range sum $p$ is at least $\frac{p-3}{2}$. Obviously, the number of roots is at most $\frac{p+3}{2}$.

Let us denote the set of roots of $f$ by $\alpha_1, \ldots , \alpha_{k} $, where $\frac{p-3}{2} \le k \le \ee$. Let $B=\{ \beta_1, \ldots, \beta_{k} \}$ be a multiset such that $\beta \in B$ if and only if $f(\beta)>1$ and the multiplicity of $\beta$ in $B$ is $f(\beta)-1$. Since the range sum of $f$ is $p$ we have that   $k$ is equal to the number of roots of $f$.

%We prove the following lemma, which is similar to Equation (2) in $\cite{negyen}$.
%\begin{lemma}\label{lem:alapegyenlet}
%Let $f=a_{\frac{p+1}{2}}x^\ee+a_\dd x^\dd +\ldots a_1x+a_0$ be a polynomial of range sum $p$. Let $c(\gamma)$ denote an element of $\{0,1,\ldots , p-1 \} \subset \mathbb{Z}$, which is congruent to $\frac{p+1}{2} a_\ee \gamma - a_\dd$.

%Then the following equality of integers holds:
%   $$ \sum_{i=1}^{\frac{p-1}{2}}\left(\frac{\alpha_i-\gamma}{p} \right) =   \sum_{i=1}^{\frac{p-1}{2}}\left( \frac{\beta_i-\gamma}{p} \right) + r(\gamma), $$
%   where $r(\gamma)= c(\gamma)  \mbox{ or } c(\gamma)-p$.
%\end{lemma}
%\begin{proof}
%    We first calculate $\sum_{x \in \F_p} f(x-\gamma)x^\dd$ in two different ways.

 %   It does follow from Lemma \ref{lem:sumandleadingcoeff} that $-\sum_{x \in \F_p} f(x-\gamma)x^\dd $ is congruent to the leading coefficient of the polynomial, which is equal to $-a_\ee \ee \gamma +a_\dd$. On the other hand, just as in \cite{negyen} we have
  %  $$\sum_{x \in \F_p} x^l f(x-\gamma) \equiv  \sum_{x \in \F_p} (x-\gamma)^l +\sum_{i=1}^k (a_i-\gamma)^l - \sum_{i=1}^k (b_i-\gamma)^l$$
%for every positive integer $l$.

%By applying this for $l=\dd$ we obtain
 %     $$ \sum_{i=1}^{\frac{p-1}{2}}\left(\frac{\alpha_i-\gamma}{p} \right) \equiv \sum_{i=1}^{\frac{p-1}{2}}\left( \frac{\beta_i-\gamma}{p} \right) + c(\gamma) \pmod{p}.$$
%$$    -\sum_{x \in \F_p} f(x-\gamma)x^\dd %=\sum_{\left(\frac{x}{p} \right) } x^$$
%\end{proof}

We will also use the following results proved in \cite{adam}.
\begin{theorem}\label{thm:regialphaabszertekbecsles}
    For any $A \subset \mathbb F_p $ we have
\begin{displaymath}
     \sum_{\gamma \in \mathbb F_p} \bigg | \sum_{\alpha \in A } \bigg( \frac{\alpha - \gamma}p \bigg)  \bigg | \leq  \frac{1}{2}p^{\frac{3}{2}}.
\end{displaymath}
\end{theorem}

As a corollary of Theorem \ref{thm:regialphaabszertekbecsles} we obtain that the number of $\gamma \in \F_p$ such that $\sum_{\alpha \in A}\left( \frac{\alpha-\gamma}{p}\right)$ is of large absolute value, is small. More precisely we have the following. 
\begin{lemma}\label{lem:S becslese}
Let $S=\left\{ \gamma \in \F_p \mid  \left| \sum_{\alpha \in A}\left( \frac{\alpha-\gamma}{p}\right)  \right|  \ge p^\frac{2}{3} \right\}$. Then $|S| \le p^{\frac{2}{3}}$.   
\end{lemma}
This set $S$ contains those number of $\F_p$ which we considered pathological and we will usually omit them from any further considerations. Lemma \ref{lem:S becslese} that $S$ is fairly small.

Let $B$ be the multiset consisting of those elements $x$ of $\F_p$ for which $f(x) > 1$ such that the weight of $x$ in $B$ is $f(x)-1$. 
It is easy to see that $|B|$ is equal to $|A|$, which is the number of different roots of $f$. We write
\begin{displaymath}
    B = \bigcup_{j= 1}^{n} B_j,
\end{displaymath}
where $B_j$ are all the homogeneous multisets containing all elements of $A$, which are equal to some element of $B$. Let 
\begin{equation}\label{eq k_j}
    k_j := | B_j |  \mbox{ and } k_j := |B_j|.
\end{equation}
The following proposition also appears in \cite{adam} but for sake of completeness we provide a proof here.
\begin{prop}\label{thm:kibecsles}
    For the $B$ multiset it holds that:
\begin{displaymath}
     \sum_{\gamma \in \mathbb F_p} \bigg | \sum_{\beta\in B } \bigg( \frac{\beta - \gamma}p \bigg)  \bigg | \leq p \sqrt{\sum_{j =1}^{n} k_j^2}.
\end{displaymath}
\end{prop}

\begin{proof}
    We have that
\begin{displaymath}
     \sum_{\gamma \in \mathbb F_p} \bigg | \sum_{\beta\in B } \bigg( \frac{\beta - \gamma}p \bigg)  \bigg | \leq \sqrt{p}\sqrt{\sum_{\gamma \in \mathbb F_p} \bigg ( \sum_{\beta\in B } \bigg( \frac{\beta - \gamma}p \bigg) \bigg)^2 }.
\end{displaymath}
\begin{displaymath}
     \sum_{\gamma \in \mathbb F_p} \bigg ( \sum_{\beta\in B } \bigg( \frac{\beta - \gamma}p \bigg) \bigg)^2 = \sum_{\gamma} \sum_{j = 1}^{n} \bigg( \sum_{\beta \in B_j} \bigg( \frac{\beta-\gamma}{p} \bigg) \bigg)^2  +  2 \sum_{\beta_1 \neq \beta_2} \sum_{\gamma} \bigg( \frac{\beta_1 - \gamma}p \bigg) \bigg( \frac{\beta_2 - \gamma}p \bigg) \leq
\end{displaymath}
\begin{displaymath}
    \leq p \sum_{j=1}^{n} k^2_j.
\end{displaymath}
It follows from the previous calculation that 
\begin{displaymath}
\sum_{\gamma \in \mathbb F_p} \bigg | \sum_{\beta\in B } \bigg( \frac{\beta - \gamma}p \bigg)  \bigg | \leq p \sqrt{\sum_{j=1}^{n} k^2_j}.
\end{displaymath}

\end{proof}
\section{Polynomials of degree $\frac{p+1}{2}$.}
Our purpose in this section is to prove Theorem \ref{thm1}. Thus we will show that there is no polynomial in $\mathbb{F}_p[x]$ of degree $\frac{p+1}{2}$ with range sum $p$ if $p$ is a prime, which is large enough.

Assume $f$ is a polynomial of degree $\ee$ with  $\sum_{x \in \F_p} f(x) = p $. We write 
$$f(x)=a_{\frac{p+1}{2}}x^\ee+a_{\frac{p-1}{2}}x^\dd + \ldots +a_1x+a_0.$$

The coefficient of $x^\dd$ of $f(x-\gamma) $  is equal to $-\frac{p+1}{2} \cdot a_{\frac{p+1}2} \cdot \gamma + a_{\frac{p-1}2}$.
Then it directly follows from Lemma that \ref{lem:sumandleadingcoeff}
\begin{displaymath}
    \sum_{x} x^{\frac{p-1}2} f(x+\gamma) \equiv -\left( -\frac{p+1}{2} \cdot a_{\frac{p+1}2} \cdot \gamma + a_{\frac{p-1}2} \right) \pmod{p}.
\end{displaymath}
We remind that as in \cite{negyen} we denote by $A$ the set of roots of $f$ and $B$ is a multiset consisting of those elements $x$ of $\F_p$ for which $f(x)>1$ counted with multiplicity $f(x)-1$.  

It was proved in \cite{negyen} that (see equation (2)) that for every $\gamma \in \F_p$
\begin{equation}\label{eq:elso}
        \sum_{\alpha \in A} \bigg( \frac{\alpha-\gamma}{p} \bigg) \equiv \sum_{\beta \in B} \bigg( \frac{\beta-\gamma}p \bigg) + \frac{1}{2}a_{\frac{p+1}{2}} \gamma + a_{\frac{p-1}2} \pmod{p}.
\end{equation}

Using suitable affine transformations on $f$ (i.e. replace $f$ by $f(ax+b)$ for some $a\in \F_p^*$, $b\in \F_p$) we may simplify equation \eqref{eq:elso}. We first use a suitable translation on $f$ to obtain the following: 
\begin{displaymath}
     \sum_{\alpha \in A'} \bigg( \frac{\alpha-\gamma}{p} \bigg) \equiv \sum_{\beta \in B'} \bigg( \frac{\beta-\gamma}p \bigg) + \frac{1}{2}a_{\frac{p+1}{2}} \gamma \pmod{p}.
\end{displaymath}
As a result the set $A$ and the multiset $B$ are also translated by the same element of $\F_p$. Using another transformation we obtain
\begin{displaymath}
     \sum_{\alpha \in A} \bigg( \frac{\alpha- 2 a_{\frac{p+1}{2}}^{-1}  \gamma}{p} \bigg) \equiv \sum_{\beta \in B} \bigg( \frac{\beta- 2 a_{\frac{p+1}{2}}^{-1}\gamma}p \bigg) +  \gamma \pmod{p}.
\end{displaymath}
We replace $A$ by $1/2 a_{\frac{p+1}{2}} A$ and $B$ by $1/2 a_{\frac{p+1}{2}} B$. Thus
\begin{displaymath}
    \sum_{\alpha \in A} \bigg( \frac{2 a_{\frac{p+1}{2}}^{-1} \alpha- 2 a_{\frac{p+1}{2}}^{-1}  \gamma}{p} \bigg) \equiv \sum_{\beta \in B} \bigg( \frac{2 a_{\frac{p+1}{2}}^{-1}\beta- 2 a_{\frac{p+1}{2}}^{-1}\gamma}p \bigg) +  \gamma \pmod{p},
\end{displaymath}
\begin{displaymath}
    \sum_{\alpha \in A} \bigg( \frac{2 a_{\frac{p+1}{2}}^{-1}}p \bigg) \bigg( \frac{ \alpha-   \gamma}{p} \bigg) \equiv \sum_{\beta \in B} \bigg( \frac{2 a_{\frac{p+1}{2}}^{-1}}p \bigg) \bigg( \frac{\beta- \gamma}p \bigg) +  \gamma \pmod{p}.
\end{displaymath}
We simplify with the nonzero common term to obtain
\begin{displaymath}
     \sum_{\alpha \in A} \bigg( \frac{ \alpha-   \gamma}{p} \bigg) \equiv \sum_{\beta \in B}  \bigg( \frac{\beta- \gamma}p \bigg) \pm \gamma  \pmod{p}.
\end{displaymath}
Notice that $\pm$ does not depend on the value of $\gamma$, only on the coefficient $\alpha_{\frac{p+1}2}$. First let
\begin{displaymath}
     \sum_{\alpha \in A} \bigg( \frac{ \alpha-   \gamma}{p} \bigg) \equiv \sum_{\beta \in B}  \bigg( \frac{\beta- \gamma}p \bigg) + \gamma.
\end{displaymath}
\begin{displaymath}
     \sum_{\alpha \in A} \bigg( \frac{ \alpha-   \gamma}{p} \bigg) - \gamma \equiv \sum_{\beta \in B}  \bigg( \frac{\beta- \gamma}p \bigg) .
\end{displaymath}
The equation above can be lifted up to the ring of integers. We lift up the Legendre symbols $(\frac{x}p)$  to the set of $\{-1,0,1\} \subset \Z$. This means that
\begin{displaymath}
    -\frac{p+1}2 \leq -|A| \leq \sum_{\alpha \in A} \bigg( \frac{\alpha -  \gamma}p \bigg) \leq |A| \leq \frac{p+1}2,
\end{displaymath}
and 
\begin{displaymath}
   - \frac{p+1}2 \leq \sum_{\beta \in B} \bigg( \frac{\beta -  \gamma}p \bigg) \leq \frac{p+1}2.
\end{displaymath}
Furthermore, the number of different quadratic residues is $\frac{p-1}2$ and the number of quadratic non-residues is also $\frac{p-1}2$ so if $|A|= \frac{p-3}{2} \mbox{ or } \ee$, then 
\begin{displaymath}
    -\frac{p-3}2 \leq \sum_{\alpha \in A} \bigg( \frac{\alpha -  \gamma}p \bigg) \leq \frac{p-3}2.
\end{displaymath}
We again identify the elements of $\F_p$ with $\{0,1,\ldots, p-1\}$. Let $r:\mathbb{F}_p=\{0,1,\ldots, p-1\} \to \mathbb{Z}$ such that 
\begin{equation}\label{eq:foegyenlet}
     \sum_{\alpha \in A} \bigg( \frac{ \alpha-   \gamma}{p} \bigg) = \sum_{\beta \in B}  \bigg( \frac{\beta- \gamma}p \bigg) + r(\gamma).
\end{equation}
holds for every $\gamma \in \mathbb{F}_p$. 
It follows from the previous simple arguments on the absolute value of these sums  in the previous equation that $|r(\gamma)| \le 2p$. Moreover, it is easy to that equation can only occur if $x^\dd-1$ divides the polynomial $f$ in $\F_p[x]$. It is not too difficult to see that the range sum of such polynomials of degree $\ee$ is larger than $p$.
Thus we may assume
that $r(\gamma)$ is $ -p,0 \mbox{ or }p$  if $\gamma = 0$, and it is equal to $\gamma$ or $\gamma -p$ if $\gamma \neq 0$. 

By applying the triangle inequality and Theorem \ref{thm:regialphaabszertekbecsles} we obtain. 
\begin{displaymath}
    \sum_{\gamma \in \mathbb{F}_p} \bigg| \sum_{\alpha \in A } \bigg( \frac{\alpha-\gamma}p\bigg) - r(\gamma)\bigg| \geq \sum_{\gamma \in \F_p} |r(\gamma)|  - p^{\frac32} \geq \frac{p^2}{4} - \frac{1}{2}p^{\frac32}.
\end{displaymath}
By the previously used $k_i$ multiplicities  of the elements of the multiset $B$ we get
\begin{displaymath}
    p \sqrt{\sum k^2_j} \geq \frac{p^2}4 -p ^{\frac32}.
\end{displaymath}
We may suppose that $k_1 \geq k_2 \geq \ldots \geq k_n$. 

\begin{lemma}
 $k_1 \geq \frac{|B|}{5}$. Thus there is $\beta' \in B$ whose multiplicity in $B$ is at least $\frac{p-1}{10}$. 
\end{lemma}

\begin{proof}
    First suppose that the statement of lemma is false. Then
\begin{displaymath}
    p \sqrt{\sum k^2_j} \leq p \sqrt{5\bigg(\frac{|B|}5\bigg)^2} = \frac{p^2}{2 \sqrt 5} < \frac{p^2}4 - p^{\frac32},
\end{displaymath}
which is a contradiction.

\end{proof}
Let 
$$
\Gamma_{-p}=\{ \gamma \in \F_p \colon r(\gamma) = \gamma-p\}.
$$
\begin{lemma}\label{lem:32}
    Let $\Gamma'$ denote the intersection of 
    $ [\frac{9p}{20}, \frac{19p}{40}] $ with $\Gamma_{-p}$. Then $\left| \Gamma' \right| \le 21\sqrt{p}$.
    %In the interval $ [\frac{9p}{20}, \frac{19p}{40}] $ the number of $\gamma$ in $\F_p$  such that $r(\gamma) = \gamma - p$ is at most $21\sqrt{p}$.
\end{lemma}

\begin{proof}
   $B$ has at most $\ee$ elements. Therefore $\sum_{\beta \in B} (\frac{\beta-\gamma}p) \leq \frac{p+1}{2}$.
   If $\gamma \in \Gamma'$, then clearly 
   $\sum_{\alpha \in A}(\frac{\alpha - \gamma}{p}) > \frac{p}{40} - \frac12$ holds. We sum all equation \eqref{eq:foegyenlet} when $\gamma$ is in $\Gamma'$.
\begin{equation}\label{eq:lem32}
    \sum_{\gamma \in \Gamma'} \sum_{\alpha \in A} \bigg( \frac{\alpha-\gamma}p\bigg) > \frac{p}{40} \frac{\lvert \Gamma' \lvert} 2   \end{equation}
Theorem \ref{thm:regialphaabszertekbecsles} gives an upper bound for the left hand side of equation \eqref{eq:lem32} so it can be estimated from above by $\frac{p^\frac{3}{2}}{2}$. Rearranging these inequalities gives the result 
\end{proof}
It is well-known that the sequence of quadratic residues is quasi-random. A more precise version of this statement is the Pólya-Vinogradov theorem. In our special case, this is expressed in the following Lemma. 
\begin{lemma}\label{lem:33}
    The number of elements $\beta$ of the interval $ [\frac{9p}{20}, \frac{19p}{40}] $ such that $\beta-\gamma$ is a quadratic residue is at least $\frac{p}{80}- \frac{1}{2}\sqrt{p}\log{p}$ for every $\beta \in \F_p$.
\end{lemma}
%\begin{proof}
%    It follows from the Pólya-Vinogradov theorem.
%\end{proof}

As a corollary of Lemma \ref{lem:32} and Lemma \ref{lem:33} we obtain the following.
\begin{cor}\label{cor:34}
In the interval of  $ [\frac{9p}{20}, \frac{19p}{40}] $ the number of elements $\gamma$ of $\F_p$ such that $\left( \frac{\beta'-\gamma}p \right) = 1 $  and $r(\gamma) = \gamma$ is at least $\frac{p}{80}- \frac{1}{2}\sqrt{p}\log{p}- 21\sqrt{p}$.
\end{cor}

%\begin{proof}
 %    For the $\gamma$ values with the property above holds that $\sum_{\beta \in B}\left( \frac{\beta-\gamma}p \right) \geq - \frac{p+1}{2} +\frac p{10} $. From this follows that for such $\gamma$ values it holds that $\sum_{\alpha}(\frac{\alpha-\gamma}p) > \frac{p}{20}$. For a $p$ prime great enough if we sum these inequalities for the $\gamma$ values we get a contradiction with the previous estimations.
%\end{proof}
If $\left( \frac{\beta'-\gamma}{p}\right)=1$, %and $\gamma \in [\frac{9p}{20}, \frac{19p}{40}] $, 
then $\sum_{\beta \in B}\left( \frac{\beta-\gamma}p \right) \geq - \frac{p+1}{2} +\frac p{10} $. Now if $\gamma \in [\frac{9p}{20}, \frac{19p}{40}] $ and $r(\gamma)=\gamma$, then 
by $\sum_{\alpha \in A} \left( \frac{\alpha-\gamma}{p}\right)= \sum_{\beta-\gamma \in B} \left( \frac{\alpha}{p}\right)+r(\gamma)$ we obtain
$\sum_{\alpha \in A}(\frac{\alpha-\gamma}p) > \frac{p}{20}$.

By Corollary \ref{cor:34} this inequality holds 
for $\frac{p}{80}- \frac{1}{2}\sqrt{p}\log{p}- 21\sqrt{p}$ elements of $\gamma$, contradicting Theorem \ref{thm:regialphaabszertekbecsles} if $p$ is large enough.

\section{Examples}
It is easy to find polnyomials of degree $\frac{p+1}{2}$ for certain small primes showing that in this case, the condition that $p$ is large enough is essential. 
\begin{itemize}
    \item $x*(x-1)*(x-2)$ for $p=5$,
    \item $x*(x-1)*(x-2)*(x-3)$ for $p=7$,
    \item $2 x*(x-1)*(x-3)*(x-5)*(x-7)*(x-9)$ for $p=11$, 
\item 
$x*(2-x)*(4-x)*(6-x)*(7-x)*(8-x)*(10-x)$ for $p=13$
\end{itemize}
On the other hand, a straightforward generalization of the examples of polynomials of degree $\dd$ of range sum $p$ would be the following. 

Assume $p=3k+1$. Then the following polynomials are of range sum $p$.
$$ 1+ \alpha x^{\frac{p-1}{3}}+ \alpha^2 x^{\frac{p-1}{3}}, 
$$
where $\alpha^3=1$. It is clear that if we multiply these polynomials with $c_1$ and $c_2$, which are solutions of $3x \equiv 1 \pmod{p}$ and 
$3x \equiv 2 \pmod{p}$, respectively, then the range sum will still be equal to $p$. 

It would be tempting to conjecture that there is no other such polynomial but this is not the case. Clearly, the range sum of $(\frac{p+1}{3} 1+  x^{\frac{p-1}{3}}+  x^{\frac{p-1}{3}} ) + \frac{p+1}{2}(x^\dd+1)$ is equal to $p$ as well, which shows that it will be more difficult to classify 'large degree' polynomials of range sum to prove combinatorial results such as the one of Gács.

\end{document}